\documentclass[11pt,reqno]{amsart}
\usepackage[utf8]{inputenc}
\usepackage[T1]{fontenc}
\usepackage{amsmath,amssymb,amsthm,mathtools}
\usepackage{mathrsfs}
\newcommand{\W}{\mathscr{W}}
\usepackage[margin=1.1in]{geometry}
\usepackage[colorlinks=true, linkcolor=cyan, citecolor=cyan]{hyperref}
\usepackage{thmtools}
\usepackage[capitalise,noabbrev]{cleveref}
\crefname{theoremA}{Theorem}{Theorems}

\newtheorem{theorem}{Theorem}[section]
\newtheorem{proposition}[theorem]{Proposition}
\newtheorem{lemma}[theorem]{Lemma}
\newtheorem{corollary}[theorem]{Corollary}
\theoremstyle{definition}
\newtheorem{definition}[theorem]{Definition}

\newtheorem{convention}[theorem]{Convention}

\theoremstyle{remark}
\newtheorem{remark}[theorem]{Remark}

\newcommand{\NN}{\mathbb{N}}
\newcommand{\ZZ}{\mathbb{Z}}
\newcommand{\QQ}{\mathbb{Q}}
\newcommand{\Qp}{\mathbb{Q}_p}
\newcommand{\RR}{\mathbb{R}}
\newcommand{\CC}{\mathbb{C}}
\newcommand{\ct}{\CC(\!(t)\!)}

\newcommand{\into}{\hookrightarrow}
\DeclareMathOperator{\ev}{ev}
\DeclareMathOperator{\Ev}{Ev}

\newcommand{\rt}[2][]{\sqrt[\leftroot{2}\uproot{3}{#1}]{#2}}

\title[Holomorphy rings and the Nullstellensatz over large fields]{Holomorphy rings and the
Nullstellensatz over large fields}

\author{Pablo Cubides Kovacsics}
\address{Departamento de Matem\'aticas, Universidad de los Andes, Bogot\'a, Colombia}
\email{pcubidesk@uniandes.edu.co}

\author{Sebasti\'an Uribe Rinc\'on}
\address{Departamento de Matem\'aticas, Universidad de los Andes, Bogot\'a, Colombia}
\email{sa.uribe@uniandes.edu.co}


\begin{document}

\begin{abstract}
For an infinite field $K$ and an ideal $I\subseteq K[X_1,\dots,X_n]$, let $J(Z_K(I))$ be the ideal of
polynomials vanishing on the $K$-rational zero locus of $I$. A Nullstellensatz over $K$ is 
an algebraic description of this closure operator on ideals. Hilbert's classical theorem provides 
such a description via the radical of $I$ over any algebraically closed field. In the present article, 
we introduce a unifying framework for Nullstellensätze over arbitrary rings using the notion of an \emph{$A$-radical}, for $A$ is a subring of $K(X)$ containing $K[X]$.  
We show that a single canonical holomorphy ring gives a Nullstellensatz over \emph{every} infinite field, recovering Hilbert's theorem when $K$ is algebraically closed. 
In addition, we give a new characterization of large fields (in the sense of Pop) via a particular Nullstellensatz: a perfect field is large exactly when the \emph{rational} holomorphy ring $A_n(K)$
---the intersection of all
$K$-rational places of residue field exactly $K$---satisfies the corresponding Nullstellensatz. Finally we 
study the order structure of the class of rings $A$ for which the $A$-radical gives a Nullstellensatz. We show that over non-algebraically closed fields this class has no minimal element, showing there is no optimal Nullstellensatz over such fields. 
\end{abstract}

\maketitle

\section{Introduction}\label{sec:intro}

In these new times where LLM's are shaping most of the mathematical landscape, this is perhaps the authors last attempt to provide a contribution more than 80\% human to mathematics. About two years ago, we started thinking about how to obtain Nullstellensätze over non-algebraically closed fields. Let us formalize this question. Let $K$ be a field, $X=(X_1,\dots,X_n)$ a tuple of indeterminates, $I\subseteq K[X]$ an ideal and $B$ be a subset of $K^n$.
Write
\[
  Z_K(I)=\{a\in K^n: f(a)=0\ \forall f\in I\}\subseteq K^n
\]
for the $K$-rational zero locus of $I$ and 
\[
  J(B)=\{g\in K[X]: g|_{B}=0\}\subseteq K[X]
  \]
for the ideal of polynomials vanishing on $B$. Hilbert's
Nullstellensatz \cite{hilbert1893} solves, for $K$ algebraically closed, the problem of describing the closure
operator $I\mapsto J(Z_K(I))$ algebraically:
\begin{equation}\label{eq:hilbert-intro}
  J(Z_K(I))=\sqrt I=\{g\in K[X]:\exists\,m\geq 1,\ g^m\in I\}.
\end{equation}
Over a field that is not algebraically closed the equality \eqref{eq:hilbert-intro} fails as $J(Z_K(I))$ is typically strictly larger than $\sqrt I$. The guiding
question of this paper is:

\begin{quote}
  \emph{Which algebraic operations on ideals describe the closure operator $I\mapsto J(Z_K(I))$ over an
  arbitrary field $K$?}
\end{quote}

Two classical answers beyond the algebraically closed case are the real Nullstellensatz
of Krivine and Stengle \cite{krivine,stengle} which
identifies $J(Z_\RR(I))$ with the real radical, and the $p$-adic Nullstellensatz of
Jarden--Roquette \cite{jardenroquette} (see also \cite{prestelroquette,srhir}), which describes $J(Z_{\QQ_p}(I))$ through the Kochen operator.
Abstractly, and over an arbitrary base field, the problem was studied by Laksov \cite{laksov1987} (after
Adkins--Gianni--Tognoli \cite{adkins-gianni-tognoli}, Dubois \cite{dubois1970} and Risler
\cite{risler1970}), who expressed $J(Z_K(I))$ as a ``$K$-radical'' (written $\rt[K]{I}$) defined through auxiliary witness
forms.

We present in this paper a unifying framework recovering some of the previous theorems by
identifying $J(Z_K(I))$ with an \emph{$A$-radical}. For a subring $K[X]\subseteq A\subseteq K(X)$ and an ideal $I\subseteq K[X]$, put
\begin{equation}\label{eq:Arad-intro}
  \rt[A]{I}:=\bigl\{g\in K[X]:\exists\,m\geq 1,\ g^m\in IA\bigr\}\;=\;K[X]\cap\sqrt{IA},
\end{equation}
where $IA$ denotes the ideal of $A$ generated by $I$ and $\sqrt{IA}$ its radical in $A$.
Concretely, for $I=(f_1,\dots,f_r)$, a polynomial $g$ lies in $\rt[A]{I}$ if and only if
$g^m=\sum_{j}\omega_j f_j$ for some $m\geq 1$ and $\omega_1,\dots,\omega_r\in A$. This is a radical ideal of $K[X]$ and $\rt[{K[X]}]{I}=\sqrt I$ recovers the classical
radical. A \emph{Nullstellensatz}, in this framework, is the assertion that
$J(Z_K(I))=\rt[A]{I}$ for a given ring $A$ and every ideal $I$; we write
$A\in\W_{K,n}$ when this holds (or simply $A\in\W_{K}$ when the number of variables is clear from the context). We will put special attention to rings $A$ which are \emph{holomorphy rings}, that is, intersections
\[
  A=\bigcap_{\varphi\in\mathcal V}\mathcal O_\varphi
\]
of valuation rings of $K(X)$ indexed by a family $\mathcal V$ of $K$-rational places---i.e.\ places
that are the identity on $K$ and finite on $K[X]$. The denominators allowed in $A$ are exactly the
rational functions that these places send to a residue in $K$; the passage from Hilbert to the real and
$p$-adic theorems is precisely an enlargement of the class of admissible denominators, dictated by the
arithmetic of $K$.

We will focus on two holomorphy rings. The larger is the \emph{evaluation ring} $E_n(K)$,
cut out by the most elementary places of all---the iterated single-variable evaluations
$P_{a,\sigma}\colon f/g\mapsto f(a)/g(a)$ taken in a chosen order $\sigma$ of the variables. The
smaller is the \emph{rational holomorphy ring}
\[
  A_n(K):=\bigcap_{\varphi}\mathcal O_\varphi,\qquad
  \varphi\ \text{ranging over all $K$-rational places of $K(X)$ of residue field \emph{exactly} $K$,}
\]
the intersection of \emph{every} $K$-rational place centred at a point of $K^n$. One always has
$A_n(K)\subseteq E_n(K)$; over $\RR$ the ring $A_n(K)$ is a relative version of
Becker's real holomorphy ring, and over $\Qp$ it is the localized Kochen ring of Jarden--Roquette.

\subsection*{Main results}
Our first theorem is that the \emph{larger} ring $E_n(K)$ computes the Nullstellensatz over an
arbitrary infinite field, with no tameness hypothesis whatsoever, and that it degenerates to $K[X]$
exactly in the classical case---so Hilbert's theorem is the algebraically closed instance.

\begin{restatable}{theoremA}{thmGeneralNSS}\label{thmA}
  For every infinite field $K$ and every $n\geq 1$, the evaluation holomorphy ring $E_n(K)$ lies in
  $\W_{K,n}$: for every ideal $I\subseteq K[X_1,\dots,X_n]$,
  \[
    J(Z_K(I))=\rt[E_n(K)]{I}.
  \]
\end{restatable}

Moreover $E_n(K)=K[X]$ when $K$ is algebraically closed (\Cref{thm:Ev-equals-poly}), so
\eqref{eq:hilbert-intro} is the special case $K=\bar K$ and Hilbert's theorem is recovered.
The proof isolates a single mechanism that we call the
\emph{domination property} $\mathrm{(D)}_n$, and shows that any subring $K[X]\subseteq A\subseteq
E_n(K)$ with $\mathrm{(D)}_n$ already lies in $\W_{K,n}$, and the evaluation ring itself has
$\mathrm{(D)}_n$ unconditionally (\Cref{cor:Ev-has-Dn}). \Cref{thmA} makes Laksov's abstract $K$-radical explicit: for $K$ infinite,
$\rt[K]{I}=\rt[E_n(K)]{I}$, the auxiliary witness forms being replaced by iterated evaluation
places.

The heart of the paper concerns the \emph{smaller} ring $A_n(K)$. Passing from $E_n(K)$ to $A_n(K)$ is the
step that produces a sharper statement, but---unlike \Cref{thmA}---it does not hold
over every field. Our main result shows that the class of fields for which $A_n(K)$ computes the Nullstellensatz is exactly the class of large fields 
in the sense of Pop (every smooth $K$-curve with a
$K$-rational point has infinitely many). This gives a new characterization of large fields in terms of the Nullstellensatz.

\begin{restatable}{theoremA}{thmCharacterization}\label{thmB}
  Let $K$ be an infinite \emph{perfect} field. The following are equivalent:
  \begin{enumerate}
    \item[(i)] $K$ is large;
    \item[(ii)] $A_n(K)\in\W_K$ for every $n\geq 1$, i.e.\ $J(Z_K(I))=\rt[A_n(K)]{I}$ for all $I$ and all
          $n$;
    \item[(iii)] $A_2(K)\in\W_K$.
  \end{enumerate}
\end{restatable}

That $A_2(K)\in\W_K$ implies $K$ being large needs no perfectness. The implication (i)$\Rightarrow$(ii) is proved 
by an ultraproduct argument. The failure (iii)$\Rightarrow$(i), and hence the necessity of largeness, 
is witnessed concretely over $K=\QQ$
by the rank-zero elliptic curve $E:Y^2=X^3-X$. Its finitely many rational points all lie on the line
$\{Y=0\}$, so with $p=Y^2-X^3+X$ and $g=Y$ one has $Z_\QQ(p)\subseteq Z_\QQ(g)$; yet $g^m/p\notin A_2(\QQ)$
for every $m\geq1$, because the smooth $\QQ$-rational branch of $E$ at the origin carries no rational
point beyond its centre. The argument generalizes to any non-large field: an irreducible plane curve
with a smooth $K$-point but only finitely many $K$-points yields, along such a branch, a rational
function witnessing the failure of $\mathrm{(D)}_2$.

The three tame fields of the classical theory---algebraically closed, real closed and $p$-adically closed---are all large, so \Cref{thmB} offers a new perspective on these theorems. This specialization is not, in general, a reproof of
the classical theorems but rather an alternative, holomorphy-ring description of $J(Z_K(I))$: over an algebraically closed field it does recover Hilbert's Nullstellensatz, over $\Qp$ it coincides
with the Jarden--Roquette identification of $A_n(\Qp)$ with the localized Kochen ring, whereas over $\RR$ it
is a presentation of the real radical genuinely different from---and as far as we can see not yielding---the Krivine--Stengle
Nullstellensatz. 

Finally, we study the class $\W_{K,n}$ itself as a poset (\Cref{sec:structure}). 

\begin{restatable}{theoremA}{thmTrichotomy}\label{thmC}
Let $K$ be an infinite field and $n\geq 1$.
  \begin{enumerate}
    \item[(1)] If $n=1$ or $K$ is algebraically closed, then $\W_{K,n}=\{E_n(K)\}$.
    \item[(2)] If $n\geq 2$ and $K$ is not algebraically closed, then $\W_{K,n}$ has no minimal
      elements---not even among integrally closed subrings---and no maximum.
  \end{enumerate}
\end{restatable}

Item (2) is the one we find most surprising: it says that there is no optimal Nullstellensatz over non-algebraically closed fields in more than one variable, even if we restrict to integrally closed rings. 

\subsection*{Relation to previous work}
The problem of computing $J(Z_K(I))$ over an arbitrary field was treated abstractly by Laksov
\cite{laksov1987}: for $k\subseteq K$ he identified $J(Z_K(I))$ with what he called a ``$K$-radical $\rt[K]{I}$''
defined through certain homogeneous witness forms whose $K$-zeros are confined to a hyperplane.
Laksov's $K$-radical is, in his words, ``an abstract set that is not easy to determine.''
\Cref{thmA} gives it a valuation-theoretic incarnation, $\rt[K]{I}=\rt[E_n(K)]{I}$, with those
witness forms replaced by iterated evaluation places. Laksov, following a problem of Tognoli, asked
whether the witness set may be shrunk to the anisotropic forms, and refuted the naive guess over
$\mathbb F_2$ and $\QQ$; Laksov and Westin \cite{laksov-westin1990} recast this as four conjectures,
three later disproved by Ludhani \cite{ludhani2025}. The structural results of
\Cref{sec:structure}---the absence of minimal elements of $\W$ outside the algebraically closed case---can be seen as the holomorphy-ring
counterpart of this circle of questions. 


\subsection*{AI Disclaimer} Although the authors of the present paper did discuss with LLM models, the driving questions, main ideas and core mathematics here presented have been proudly humanly produced. This project started almost 2 years ago, way before LLM's got their current superpowers. Proving a Nullstellensatz for the field of complex Laurent series (see \cite{cubides-uribe}) was the initial motivation leading to this article. Let us now give explicit credit to the parts where LLM's were involved: 
\begin{itemize}
    \item \Cref{thm:Q-fails-domination} was first proved by us. Then Claude Opus 4.8 suggested to the first author that the proof could be extended to every non-large field, that is, \Cref{thm:nonlarge-fails-domination}. The proof was a bit convoluted and the second authors with the help of Deepseek V4 found a better presentation closer to the current one.      
    \item Early proofs of Theorems \ref{thm:no-minimal} and \ref{thm:no-maximum} where also suggested by Claude Opus 4.8. The current proofs have been modified and simplified by us using our own technics.  
    \item Both authors checked the soundness of the present arguments with the cited models. 
\end{itemize}

\subsection*{Organization and conventions}
\Cref{sec:setup} sets up the framework: holomorphy rings and evaluation places, the
$A$-radical, and the witness lemma. \Cref{sec:general} proves
\Cref{thmA} via the domination property and recovers Hilbert. \Cref{sec:characterization}
proves \Cref{thmB}; the classical real, $p$-adic and power-series instances are discussed in
\Cref{rmk:instances}. \Cref{sec:structure} studies the poset $\W_{K,n}$: the trichotomy for its minimal
elements (none over a non-algebraically-closed field for $n\geq 2$, even among integrally closed
rings), the failure of a maximum over $\RR$, and the failure of
the unlocalized Kochen-type presentation of $A_n(\RR)$.

Throughout, as noted before, $K$ is a field and $X=(X_1,\dots,X_n)$. For an ideal $I\subseteq K[X]$ and a subset
$Y\subseteq K^n$, the zero locus and the vanishing ideal are
\[
  Z_K(I)=\{a\in K^n : f(a)=0\ \ \forall f\in I\},\qquad
  J(Y)=\{g\in K[X] : g(a)=0\ \ \forall a\in Y\}.
\]
Recall that the operator $I\mapsto J(Z_K(I))$ is a closure operator on the ideals of $K[X]$ (extensive,
monotone, and idempotent); in that respect, an ideal $I$ is called \emph{closed} if $I=J(Z_K(I))$.

\section{Setup and preliminaries}\label{sec:setup}

A Nullstellensatz, in the terminology of \Cref{sec:intro}, is a description of the closure $J(Z_K(I))$
by means of an algebraic operation on $I$; this section introduces the rings that carry it out.

\subsection{Holomorphy rings and evaluation places}

\begin{definition}\label{def:holomorphy-ring}
  A subring $A\subseteq K(X)$ with $K[X]\subseteq A$ is a \emph{holomorphy ring of $K(X)$ over
  $K[X]$} if it admits a presentation
  \[
    A \;=\; \bigcap_{v\in\mathcal V}\mathcal O_v
  \]
  as an intersection of valuation rings $\mathcal O_v\subseteq K(X)$, with $\mathcal V$ a family
  of valuations of $K(X)$ such that $K[X]\subseteq\mathcal O_v$ for every $v\in\mathcal V$. (Each
  such $v$ is automatically trivial on $K$: since $K\subseteq K[X]\subseteq\mathcal O_v$ and $K$ is
  a field, $v(k)=0$ for every $k\in K^\times$.)
\end{definition}

The term \emph{holomorphy ring} was introduced by Roquette \cite{roquette1973}, by analogy with rings
of functions holomorphic (that is, pole-free) at a prescribed set of places. The simplest holomorphy ring is the ring of those
rational functions whose denominator does not vanish on $K^n$, which we now define:

\begin{definition}\label{def:S} We define the \emph{ring of regular functions on
  $K^n$} as
  \[
    \mathcal R_n(K) \;:=\; \bigl\{f/g\in K(X) : f,g\in K[X],\ Z_K(g)=\emptyset\bigr\}.
  \]
\end{definition}

\begin{remark}
The ring $\mathcal R_n(K)$ is also the localization of $K[X]$ at the multiplicative set of
  polynomials without a $K$-rational zero. For $K=\RR$ this is the ring $\mathcal R(\RR^n)$ of \cite[Def.\ 3.2.1]{bcr}. Note that $\mathcal R_n(K)$ is indeed a holomorphy ring of $K(X)$ over $K[X]$: it is the intersection of the valuation rings $\mathcal O_v$ whose maximal ideal contains some irreducible factor of the reduced denominator $g$ of $f/g\in K(X)$.
\end{remark}

A finer family of holomorphy rings arises from the most elementary places of $K(X)$ over $K[X]$,
those obtained by iterated evaluation in single variables. Following \cite[\S 1]{srhir}, we single
out this class.

\begin{definition}\label{def:eval-places}
  For $b\in K$ and $1\leq l\leq n$, write a rational function $f/g\in K(X)$ in \emph{$l$-coprime
  form}, i.e.\ with $\gcd(f,g)$ trivial when regarded in $K(X_{\neq l})[X_l]$. The
  \emph{evaluation place at $(b,l)$} is the map
  \[
    P_{b,l}\colon K(X)\cup\{\infty\} \longrightarrow K(X_{\neq l})\cup\{\infty\},\qquad
    \frac{f}{g}\;\longmapsto\;\begin{cases}\dfrac{\ev_{b,l}(f)}{\ev_{b,l}(g)} & \text{if }\ev_{b,l}(g)\neq 0,\\[3pt] \infty & \text{otherwise,}\end{cases}
  \]
  where $\ev_{b,l}$ substitutes $X_l\mapsto b$ and $P_{b,l}(\infty)=\infty$. For $a=(a_1,\dots,a_n)\in K^n$ and a permutation
  $\sigma\in S_n$, define the \emph{$\sigma$-iterated evaluation}
  \[
    P_{a,\sigma} \;=\; P_{a_{\sigma(n)},\sigma(n)}\circ\cdots\circ P_{a_{\sigma(1)},\sigma(1)}
    \colon K(X)\cup\{\infty\} \longrightarrow K\cup\{\infty\}.
  \]
  Write $\Ev_{K,n}=\{P_{a,\sigma}:a\in K^n,\sigma\in S_n\}$.
\end{definition}

\begin{definition}[Evaluation holomorphy ring]\label{def:Ev}
  \[
    E_{n}(K) \;:=\; \bigl\{z\in K(X) : P_{a,\sigma}(z)\neq\infty\text{ for every }P_{a,\sigma}\in \Ev_{K,n}\bigr\}
    \;=\;\bigcap_{P\in \Ev_{K,n}}\mathcal{O}_P,
  \]
  the holomorphy ring of $K(X)$ over $K[X]$ associated to the family $\Ev_{K,n}$ of iterated
  evaluation places.
\end{definition}

\begin{remark}\label{rmk:Ev-valuative}
  Each $\mathcal O_{P_{a,\sigma}}:=\{z\in K(X):P_{a,\sigma}(z)\neq\infty\}$ is indeed a valuation
  ring of $K(X)$, so $E_n(K)$ is a holomorphy ring in the sense of
  \Cref{def:holomorphy-ring}: for a single evaluation, writing $f/g$ in $l$-coprime form,
  $P_{b,l}(f/g)\neq\infty$ precisely when $\operatorname{ord}_{X_l-b}(g)=0$, so
  $\mathcal O_{P_{b,l}}$ is the discrete valuation ring $K(X_{\neq l})[X_l]_{(X_l-b)}$ of $K(X)$;
  and a composite of places is a place, its valuation ring being the preimage of that of the second
  place under the residue map of the first \cite[Ch.~VI]{zariski-samuel}. In particular $E_n(K)$ is
  integrally closed in $K(X)$.
\end{remark}

\begin{proposition}\label{prop:S-in-Ev}
For any $K$ and any $n\geq 1$, we have $\mathcal R_n(K)\subseteq E_n(K)$ with equality for
$n=1$.
\end{proposition}

\begin{proof}
  If $g$ has no zeros in $K^n$, then for every $a\in K^n$ and every $\sigma$, the iterated
  evaluation $P_{a,\sigma}(g)\in K$ is nonzero (a product of nonvanishing residues), so
  $P_{a,\sigma}(f/g)\in K$ is finite.

  For $n=1$, let $f/g\in E_1(K)$ be written in lowest terms. If $g$ had a root $a\in K$,
  then $f(a)\neq 0$ by coprimality and $P_{a,1}(f/g)=\infty$, contradicting $f/g\in E_1(K)$.
  Hence $g$ has no root in $K$ and $f/g\in\mathcal R_1(K)$.
\end{proof}

The following computation illustrates that elements of $E_n(K)$ need not have nowhere-vanishing
denominators when $K$ is not algebraically closed; it is the case $d=2$ of
\Cref{thm:S-equals-Ev-iff-acf}.

\begin{proposition}\label{prop:X1-witness}
  Let $K$ be either $\Qp$ with $p\equiv 3\pmod 4$, or $\RR$. Then
  $\omega:=X_1^2/(X_1^2+X_2^2)\in E_{2}(K)$, and consequently
  $X_1\in\rt[E_n(K)]{(X_1^2+X_2^2)}$ via $X_1^2=\omega\cdot(X_1^2+X_2^2)$.
\end{proposition}

\begin{proof}
  The only fact about $K$ used is that $X_1^2+X_2^2$ is anisotropic over $K$, i.e.\ does not vanish
  away from the origin (for $K=\Qp$ with $p\equiv 3\pmod 4$, $-1$ is not a square; for $K=\RR$, $-1$
  is not a sum of squares). For $a\in K^2\setminus\{0\}$: $P_{a,\sigma}(\omega)=a_1^2/(a_1^2+a_2^2)\in
  K$. For $a=(0,0)$: $\omega$ is $1$-coprime and $2$-coprime, with
  $\ev_{0,1}(X_1^2+X_2^2)=X_2^2\neq 0$ giving $P_{0,1}(\omega)=0$, and
  $\ev_{0,2}(X_1^2+X_2^2)=X_1^2\neq 0$ giving $P_{0,2}(\omega)=1$. Hence $\omega\in E_{2}(K)$.
\end{proof}

\begin{remark}
  The identity $X_1^2=\omega(X_1^2+X_2^2)$ uses the exponent $m=2$, not $m=1$: indeed
  $X_1/(X_1^2+X_2^2)\notin E_n(K)$, because $P_{0,2}\bigl(X_1/(X_1^2+X_2^2)\bigr)=1/X_1$ has a pole at
  $X_1=0$ on the next evaluation. Raising to $m=2$ provides the extra vanishing required.
\end{remark}

\subsection{The $A$-radical}

For any subring $K[X]\subseteq A\subseteq K(X)$ and any ideal $I=(f_1,\dots,f_r)\subseteq K[X]$, the
$A$-radical $\rt[A]{I}$ is defined by 
\begin{equation}\label{eq:Arad-intro}
  \rt[A]{I}:=\Bigl\{g\in K[X]:\exists\,m\geq 1,\ \omega_1,\dots,\omega_r\in A,\
                  g^m=\textstyle\sum_{j}\omega_j f_j\Bigr\}=\;K[X]\cap\sqrt{IA}.
\end{equation}
It is a radical ideal of $K[X]$. Note that the function $A\mapsto\rt[A]{I}$ is
monotone, that is, $A\subseteq A'\Rightarrow\rt[A]{I}\subseteq\rt[A']{I}$ and $\rt[{K[X]}]{I}=\sqrt{I}$.

\begin{lemma}\label{lem:Arad-in-JZ}
  Let $K[X]\subseteq A\subseteq K(X)$ be a subring and $I\subseteq K[X]$ an ideal. Suppose
  that for every $a\in Z_K(I)$ there is a permutation $\sigma_a$ with
  $A\subseteq\mathcal O_{P_{a,\sigma_a}}$. Then $\rt[A]{I}\subseteq J(Z_K(I))$. In particular,
  this holds for every ideal $I$ when $A\subseteq E_n(K)$.
\end{lemma}

\begin{proof}
  Let $g\in\rt[A]{I}$, say $g^m=\sum_i a_ix_i$ with $a_i\in A$ and $x_i\in I$. For
  $a\in Z_K(I)$, the place $P_{a,\sigma_a}$ is a ring homomorphism on $A$ restricting to
  $h\mapsto h(a)$ on $K[X]$ (\Cref{rmk:Ev-valuative}), so
  $g(a)^m=\sum_i P_{a,\sigma_a}(a_i)\,x_i(a)=0$, whence $g(a)=0$. Thus $g\in J(Z_K(I))$.
\end{proof}



\subsection{The witness lemma}

The next lemma is the field-theoretic substitute for ``$Z_K(I)=\emptyset$ implies $1\in I$,'' valid
over arbitrary infinite fields and crucial for the Rabinowitsch trick in \Cref{sec:general}. It is the
weak form of the Nullstellensatz over an arbitrary field; the statement and the homogenization argument
below coincide with a result of V\'elez, G\'omez-Ram\'irez and Gallego \cite{velez-gomez-gallego}.

\begin{lemma}\label{lem:witness}
  Let $K$ be an infinite field and $J=(f_1,\dots,f_r)\subseteq K[X_1,\dots,X_n]$ an ideal with
  $Z_K(J)=\emptyset$. Then there exists $h_0\in J$ such that $Z_K(h_0)=\emptyset$.
\end{lemma}

\begin{proof}
  If $K$ is algebraically closed, the classical weak Nullstellensatz gives $1\in J$; take $h_0=1$.

  Otherwise fix a non-constant $g(Y)=a_0+a_1Y+\cdots+a_mY^m\in K[Y]$ with no root in $K$ (note
  $a_0,a_m\neq 0$). For $g_1,g_2\in K[X]$ define the \emph{homogenization}
  \[
    h_{g_1,g_2} \;:=\; g_2^m\cdot g(g_1/g_2)\;=\;a_0g_2^m + a_1g_2^{m-1}g_1+\cdots+a_mg_1^m\;\in (g_1,g_2).
  \]
  A direct case analysis shows $Z_K(h_{g_1,g_2})=Z_K(g_1)\cap Z_K(g_2)$:
  if $g_2(a)\neq 0$ then $h_{g_1,g_2}(a)=g_2(a)^m\,g(g_1(a)/g_2(a))\neq 0$ since $g$ has no root in
  $K$; if $g_2(a)=0$ and $g_1(a)\neq 0$ only $a_mg_1(a)^m\neq 0$ survives; and both vanishing trivially
  give $h_{g_1,g_2}(a)=0$.

  If $r=1$, take $h_0=f_1$. Otherwise iterate: $h_1=h_{f_1,f_2}$, $h_l=h_{h_{l-1},f_{l+1}}$ for
  $2\leq l\leq r-1$. By induction, $Z_K(h_l)=Z_K(f_1)\cap\cdots\cap Z_K(f_{l+1})$, hence
  $h_0:=h_{r-1}\in J$ satisfies $Z_K(h_0)=\emptyset$.
\end{proof}

\section{The general Nullstellensatz: $J(Z_K(I))=\rt[E_n(K)]{I}$}\label{sec:general}

This section proves \Cref{thmA}. We isolate a single mechanism, the \emph{domination property}, and show that it, 
together with containment in $E_n(K)$, is all a ring $A$ needs in order to compute the closure $J\circ Z_K$ as an $A$-radical.

\subsection{The domination property}

\begin{definition}\label{def:domination-property}
  Let $K$ be an infinite field and $K[X]\subseteq A\subseteq K(X)$ a subring. We say $A$ has the
  \emph{domination property} $\mathrm{(D)}_n$ if for all $p,g\in K[X]$ with
  $Z_K(p)\subseteq Z_K(g)$ there exists an integer $m\geq 1$ with $g^m/p\in A$.
\end{definition}

The evaluation holomorphy ring has this property.

\begin{lemma}\label{lem:domination}
  Let $K$ be an infinite field and $p,g\in K[X_1,\dots,X_n]$ with
  $Z_K(p)\subseteq Z_K(g)$. Then for every integer $m\geq\max(1,\deg p)$,
  $g^m/p\in E_{n}(K)$.
\end{lemma}

\begin{proof}
  We may assume $p\neq 0$ and $m\in\ZZ_{>0}$; the case $g=0$ is trivial.
  By induction on $n$. Write $g^m/p=F/G$ in reduced form: $F=g^m/d$, $G=p/d$, $d=\gcd(g^m,p)$ in
  $K[X]$. Since $\gcd(F,G)=1$ in this UFD, $F/G$ is $l$-coprime for every $l$: if $F$ and
  $G$ had a nonconstant common factor in $K(X_{\neq l})[X_l]$, Gauss' lemma applied to
  $K[X_{\neq l}][X_l]$ would yield a nonconstant common factor in $K[X]$, contradicting
  $\gcd(F,G)=1$.

  \medskip\noindent\emph{Base case $n=1$.}
  Here $E_n(K)=\mathcal R_n(K)=\{h/q\in K(X_1):q\text{ has no root in }K\}$. Let $r_1,\dots,r_k\in K$ be the roots
  of $p$ in $K$ with multiplicities $e_1,\dots,e_k$. By hypothesis each $r_i$ is a root of $g$
  with multiplicity $d_i\geq 1$. The multiplicity of $r_i$ in the reduced denominator $G$ is
  $\max(e_i-md_i,0)$, and $m\geq\deg p\geq e_i$ together with $d_i\geq 1$ gives $md_i\geq m\geq e_i$,
  so $G$ has no roots in $K$ and $F/G\in \mathcal R_n(K)=E_n(K)$.

  \medskip\noindent\emph{Inductive step $n\geq 2$.}
  Fix $P_{a,\sigma}\in \Ev_{K,n}$ and let $l=\sigma(1)$, $b=a_{\sigma(1)}$. Let
  $\alpha=\mathrm{mult}_{X_l-b}(g)$ and $\beta=\mathrm{mult}_{X_l-b}(p)$ in $K[X]$.

  \emph{Claim: $\beta>0\Rightarrow\alpha>0$.} If $(X_l-b)\mid p$ then $p(b,X_{\neq l})\equiv 0$,
  so $\{b\}\times K^{n-1}\subseteq Z_K(p)\subseteq Z_K(g)$; since $K$ is infinite this
  forces $g(b,X_{\neq l})\equiv 0$, i.e.\ $(X_l-b)\mid g$.

  Hence $\mathrm{mult}_{X_l-b}(G)=\max(\beta-m\alpha,0)$ and
  $\mathrm{mult}_{X_l-b}(F)=\max(m\alpha-\beta,0)$.

  \emph{Case 1: $\alpha\geq 1$.} Then $m\alpha\geq m\geq\deg p\geq\beta$, so
  $\mathrm{mult}_{X_l-b}(G)=0$ and $\mathrm{mult}_{X_l-b}(F)\geq 0$. If $m\alpha>\beta$, then
  $(X_l-b)\mid F$, and $G(b,X_{\neq l})\not\equiv 0$ gives $P_{b,l}(F/G)=0$, finite. If
  $m\alpha=\beta$, then $m=\deg p$, $\alpha=1$, $\beta=\deg p$, $p=c(X_l-b)^{\deg p}$ for some
  $c\in K^\times$, and direct calculation gives $P_{b,l}(F/G)=g_0(b,X_{\neq l})^m/c\in K[X_{\neq l}]$
  where $g=(X_l-b)g_0$. All subsequent evaluations of this polynomial are finite.

  \emph{Case 2: $\alpha=0$ (hence $\beta=0$ by the claim).} Then $(X_l-b)$ divides neither $g$ nor
  $p$, so $d(b,X_{\neq l})\not\equiv 0$ and
  $P_{b,l}(F/G)=g(b,X_{\neq l})^m/p(b,X_{\neq l})$.
  The hypothesis is preserved: $a'\in K^{n-1}$, $p(b,a')=0$ $\Rightarrow$ $(b,a')\in Z_K(p)\subseteq Z_K(g)$
  $\Rightarrow$ $g(b,a')=0$. The degree bound is preserved: $\deg p(b,X_{\neq l})\leq\deg p\leq m$.
  By induction the remaining $n-1$ evaluations give a finite value.

  In all cases $P_{a,\sigma}(g^m/p)\neq\infty$, so $g^m/p\in E_n(K)$.
\end{proof}

\begin{remark}\label{rmk:bound-sharp}
  The bound $m\geq\deg g$ is sharp: for $f=X_1$, $g=(X_1^2+X_2^2)^2$ (with $\deg g=4$), one has
  $P_{(0,0),(2,1)}(X_1^4/g)=1$ (finite) but $P_{(0,0),(2,1)}(X_1^3/g)=\infty$.
\end{remark}

\begin{corollary}\label{cor:Ev-has-Dn}
  For every infinite field $K$ and every $n\geq 1$, $E_n(K)$ has the domination property
  $\mathrm{(D)}_n$.
\end{corollary}

\begin{proof}
  Immediate from \Cref{lem:domination}.
\end{proof}

\subsection{The class \texorpdfstring{$\W$}{W} and the general Nullstellensatz}

\begin{definition}\label{def:W}
  For $n\geq 1$, let $\W_{K,n}$ denote the class of subrings $A\subseteq K(X)$, $X=(X_1,\dots,X_n)$,
  such that (i) $K[X]\subseteq A$; and (ii) $\rt[A]{I}=J(Z_K(I))$ for every ideal $I\subseteq K[X]$.
  We set
  \[
    \W_K:=\bigcup_{n\geq 1}\W_{K,n}.
  \]
  When $K$ is clear from context we even abbreviate $\W_K$ by $\W$.
\end{definition}

\begin{theorem}\label{thm:abstract-membership}
Let $K$ be an infinite field and $K[X]\subseteq A\subseteq K(X)$ a subring such that
  \begin{enumerate}
    \item[(a)] $A$ has the domination property $\mathrm{(D)}_n$ and
    \item[(b)] for every $a\in K^n$ there is a permutation $\sigma_a$ with
      $A\subseteq\mathcal O_{P_{a,\sigma_a}}$.
  \end{enumerate}
  Then $A\in \W_{K,n}$.
\end{theorem}

\begin{proof}
  The inclusion $\rt[A]{I}\subseteq J(Z_K(I))$ follows from \textup{(b)}
  and~\Cref{lem:Arad-in-JZ}. For the reverse inclusion,
  let $g\in J(Z_K(I))$ with $I=(f_1,\dots,f_r)$. We apply
  Rabinowitsch's trick: in $K[X_0,X_1,\dots,X_n]$, the ideal $J=(f_1,\dots,f_r,1-X_0g)$ satisfies
  $Z_K(J)=\emptyset$, so by \Cref{lem:witness} there exists $h_0\in J$ with
  $Z_K(h_0)=\emptyset$. Write $h_0=\sum_j a_jf_j+b(1-X_0g)$ in $K[X_0,X]$ and set
  $N=\max(\deg_{X_0}h_0,\max_j\deg_{X_0}a_j)$. Substituting $X_0=1/g$ and multiplying by $g^N$
  produces $p:=g^Nh_0(X,1/g)\in I$ via $p=\sum_j\beta_jf_j$ with $\beta_j:=g^Na_j(X,1/g)\in K[X]$,
  and $Z_K(p)\subseteq Z_K(g)$ (if $g(a)\neq0$ then $p(a)=g(a)^Nh_0(a,1/g(a))\neq0$,
  since $Z_K(h_0)=\emptyset$). By $\mathrm{(D)}_n$ there is $m\geq 1$ with $g^m/p\in A$;
  hence $g^m=(g^m/p)\,p\in IA$ and $g\in\rt[A]{I}$.
\end{proof}

\begin{corollary}\label{cor:Ev-membership} 
Let $K$ be an infinite field and $K[X]\subseteq A\subseteq E_n(K)$ a subring with the
  domination property $\mathrm{(D)}_n$. Then $A\in \W_{K,n}$.
  \end{corollary}

\begin{proof}
  Condition \textup{(b)} of \Cref{thm:abstract-membership} holds with any choice of
  $\sigma_a$, since $A\subseteq E_n(K)\subseteq\mathcal O_{P_{a,\sigma}}$ for all $a$ and
  $\sigma$.
\end{proof}



\thmGeneralNSS*

\begin{proof}
  This follows by \Cref{cor:Ev-has-Dn} and \Cref{cor:Ev-membership}.
\end{proof}

\begin{remark}\label{rmk:laksov}
  For $k=K$ infinite, both $\rt[E_n(K)]{I}$ and Laksov's $K$-radical $\rt[K]{I}$
  \cite{laksov1987} coincide with $J(Z_K(I))$, so $\rt[E_n(K)]{I}=\rt[K]{I}$. The content is not the
  equality itself but the computation of the right-hand side: the abstract $K$-radical, defined
  through Laksov's witness forms, is realized by a single holomorphy ring, with those forms replaced
  by iterated evaluation places.
\end{remark}

\subsection{Recovery of Hilbert's Nullstellensatz}

\begin{theorem}\label{thm:Ev-equals-poly}
  If $K$ is algebraically closed then $E_{n}(K)=K[X_1,\dots,X_n]$.
\end{theorem}

\begin{proof}
  Let $f/g\in E_n(K)$ in lowest terms with $g$ non-constant. Since $K$ is algebraically closed,
  $\dim Z_K(g)\geq n-1$ while $\dim(Z_K(f)\cap Z_K(g))\leq n-2$, so there exists $a\in Z_K(g)\setminus Z_K(f)$.
  We show by induction on $n$ that $P_{a,\sigma}(f/g)=\infty$ for some $\sigma$, contradicting
  $f/g\in E_n(K)$.

  \emph{$n=1$:} $f/g$ is in lowest terms, $g(a_1)=0$, so $P_{a_1,1}(f/g)=\infty$.

  \emph{$n\geq 2$:} If $X_l-a_l\mid g$ for some $l$, then $g(a_l,X_{\neq l})\equiv 0$ while
  $f(a_l,X_{\neq l})\not\equiv 0$ by coprimality, so $P_{a_l,l}(f/g)=\infty$ and any $\sigma$ with
  $\sigma(1)=l$ works. Otherwise, for any $l$, $g(a_l,X_{\neq l})\not\equiv 0$, and reducing to
  lowest terms produces $f''/g''\in K(X_{\neq l})$ with $g''(a_{\neq l})=0$, $f''(a_{\neq l})\neq 0$,
  $g''$ non-constant. By induction, there exists $\sigma'$ on $\{1,\dots,n\}\setminus\{l\}$ with
  $P_{a_{\neq l},\sigma'}(f''/g'')=\infty$; composing with $P_{a_l,l}$ gives $\sigma$ with
  $P_{a,\sigma}(f/g)=\infty$.
\end{proof}

\begin{proposition}\label{thm:S-equals-Ev-iff-acf}
  Let $n\geq 2$. Then $\mathcal R_n(K)=E_n(K)$ if and only if $K$ is algebraically closed.
\end{proposition}

\begin{proof}
  ($\Leftarrow$) If $K$ is algebraically closed, every non-constant polynomial in $K[X]$ has a zero
  in $K^n$, so the only denominators occurring in $\mathcal R_n(K)$ are constants and
  $\mathcal R_n(K)=K[X]$. By \Cref{thm:Ev-equals-poly} we conclude $E_n(K)=K[X]$.

  ($\Rightarrow$) We prove the contrapositive, so suppose $K$ is not algebraically closed. We exhibit an
  element of $E_n(K)\setminus\mathcal R_n(K)$. Choose an irreducible
  $q(Y)=\sum_{i=0}^d c_iY^i\in K[Y]$ of degree $d\geq 2$ with $c_0,c_d\neq 0$. In particular, it has no root in $K$. 
  Its homogenization
  \[
    Q \;:=\; X_1^d\,q(X_2/X_1)\;=\;\sum_{i=0}^d c_i\,X_1^{d-i}X_2^i
  \]
  is a binary form whose only zero in $K^n$ is the subspace $\{X_1=X_2=0\}$: if $Q(a)=0$ with
  $(a_1,a_2)\neq(0,0)$, then $a_1\neq 0$ would force $q(a_2/a_1)=0$ and $a_1=0$ would force
  $c_da_2^d=0$, both impossible. Put $\omega:=X_1^d/Q$. Since $Q(0,X_2)=c_dX_2^d\neq 0$ and
  $Q(X_1,0)=c_0X_1^d\neq 0$, the numerator $X_1^d$ and the denominator $Q$ are coprime, so $\omega$
  is in lowest terms. We show that $\omega\in E_n(K)\setminus \mathcal R_n(K)$.
  
  As $\omega$ depends only on $X_1,X_2$, evaluating any of the remaining
  variables leaves it unchanged and never produces $\infty$, so it suffices to track the
  evaluations of $X_1$ and $X_2$. For $a\in K^n$ with $(a_1,a_2)\neq(0,0)$ the final value
  $a_1^d/Q(a_1,a_2)$ is finite, and no intermediate pole arises because $Q(a_1,X_2)$ has leading
  coefficient $c_d\neq 0$ and $Q(X_1,a_2)$ has leading coefficient $c_0\neq 0$ (so neither vanishes
  identically). For $(a_1,a_2)=(0,0)$, evaluating $X_1$ first gives $P_{0,1}(\omega)=0/(c_dX_2^d)=0$,
  while evaluating $X_2$ first gives $P_{0,2}(\omega)=X_1^d/(c_0X_1^d)=1/c_0$; both are finite. Hence
  $P_{a,\sigma}(\omega)\neq\infty$ for all $a$ and all $\sigma$, i.e.\ $\omega\in E_n(K)$.

  It remains to show that $\omega\notin\mathcal R_n(K)$. If $\omega=f/g$ in $K(X)$, then $fQ=gX_1^d$ and
  $\gcd(X_1^d,Q)=1$ force $Q\mid g$; hence $Z_K(g)\supseteq Z_K(Q)\ni 0$, so $\omega$
  admits no denominator nowhere vanishing on $K^n$.

  Therefore $\mathcal R_n(K)\subsetneq E_n(K)$. The instance $q(Y)=Y^2+1$ (giving $Q=X_1^2+X_2^2$),
  valid for $K=\RR$ and for $K=\Qp$ with $p\equiv 3\pmod 4$, is worked out in
  \Cref{prop:X1-witness}.
\end{proof}

\begin{corollary}[Hilbert's Nullstellensatz]\label{cor:hilbert}
  For $K$ algebraically closed and $I\subseteq K[X_1,\dots,X_n]$, $J(Z_K(I))=\sqrt{I}$.
\end{corollary}

\begin{proof}
  By \Cref{thmA}, $J(Z_K(I))=\rt[E_n(K)]{I}$; by \Cref{thm:Ev-equals-poly},
  $E_n(K)=K[X]$, so $\rt[E_n(K)]{I}=\sqrt{I}$.
\end{proof}

\section{A characterization of large fields via Nullstellensätze}\label{sec:characterization}

  The notion of a large field is due to Pop \cite{pop-ample}: $K$ is large
  if every smooth $K$-curve with a $K$-point has infinitely many. Equivalently $K$ is large if and only if $K$ is
  existentially closed in $K(\!(t)\!)$. The main result of this section is a characterizaction of $K$ being large in terms of a certain ring belonging to $\W_K$. 
  We first introduce the relevant ring, the rational holomorphy ring $A_n(K)$, which is a natural intermediate between the evaluation holomorphy ring $E_n(K)$ and the iterated evaluation holomorphy ring $\mathcal R_n(K)$.

\subsection{The rational holomorphy ring}\label{subsec:AK-domination}

\begin{definition}\label{def:AK}
  Let $K$ be an infinite field. Let $\mathcal S^{\mathrm{rat}}_{K}=\mathcal S^{\mathrm{rat}}_{K,n}$
  denote the set of places $\varphi$ of $K(X)$ over $K$ with $K[X]\subseteq\mathcal O_\varphi$
  (equivalently, $\varphi(X_i)\in K$ for all $i$, so $\varphi$ is centred at a point of $K^n$) and
  with residue field exactly $K$. Set
  \[
    A_n(K)\;=\;A_n(K)\;:=\;\bigcap_{\varphi\in\mathcal S^{\mathrm{rat}}_{K}}\mathcal O_\varphi .
  \]
\end{definition}

Since every iterated evaluation place $P_{a,\sigma}$ of \Cref{def:eval-places} lies in
$\mathcal S^{\mathrm{rat}}_{K}$, while $\mathcal S^{\mathrm{rat}}_{K}$ may contain further places, one has the sandwich
\[
  \mathcal R_n(K)\;\subseteq\;A_n(K)\;\subseteq\;E_{n}(K).
\]
Because $A_n(K)\subseteq E_n(K)$, \Cref{cor:Ev-membership} applies to
$A_n(K)$, and its converse holds as well: membership of $A_n(K)$ in $\W$ is governed entirely by the
domination property.

\begin{theorem}\label{thm:AK-iff-domination}
  Let $K$ be an infinite field and $n\geq 1$. Then
  \[
    A_n(K)\in \W_K\quad\Longleftrightarrow\quad A_n(K)\text{ has the domination property }\mathrm{(D)}_n.
  \]
  In either case the inclusion $\rt[A_n(K)]{I}\subseteq J(Z_K(I))$ holds unconditionally, since
  $A_n(K)\subseteq E_n(K)$ (\Cref{lem:Arad-in-JZ}).
\end{theorem}

\begin{proof}
  ($\Leftarrow$) Since $A_n(K)\subseteq E_n(K)$ and $A_n(K)$ has $\mathrm{(D)}_n$,
  \Cref{cor:Ev-membership} gives $A_n(K)\in \W_K$.

  ($\Rightarrow$) Assume $A_n(K)\in \W_K$ and let $p,g\in K[X]$ with $p\neq 0$ and
  $Z_K(p)\subseteq Z_K(g)$. Take $I=(p)$; then $g$ vanishes on $Z_K(I)=Z_K(p)$,
  so $g\in J(Z_K(I))=\rt[A_n(K)]{(p)}$. Hence $g^m=\omega\,p$ for some $m\geq 1$ and $\omega\in A_n(K)$, i.e.\
  $g^m/p=\omega\in A_n(K)$. This is $\mathrm{(D)}_n$.
\end{proof}

\Cref{thm:AK-iff-domination} reduces the entire question to the single property
$\mathrm{(D)}_n$ for $A_n(K)$. The main theorem of this section is to show that such a property will fail over non-large fields and holds over large ones.

\subsection{Failure of domination over non-large fields}\label{subsec:Q-failure}

Let us first show a concrete counterexample in the case $K=\QQ$, showing that $A_2(\QQ)$ fails $\mathrm{(D)}_2$. We realize this with a rank-zero elliptic curve. Both this counterexample and its generalization below rest on the following technical lemma. 

\begin{lemma}\label{lem:technical}
  Let $K$ be a field and $p\in K[X,Y]$ with $p(0,0)=0$ and $\partial_Xp(0,0)\neq0$. Then the
  substitution homomorphism
  \[
    \psi\colon K[[Y,t]]\longrightarrow K[[X,Y]],\qquad Y\mapsto Y,\quad t\mapsto p(X,Y),
  \]
  is an isomorphism of $K$-algebras. In particular, $\varphi:=\psi^{-1}$ satisfies $\varphi(Y)=Y$
  and $\varphi(p)=t$, and restricts to an embedding of $K[X,Y]$ into $K[[Y,t]]$.
\end{lemma}

\begin{proof}
  Write $\mathfrak m=(X,Y)\subseteq K[[X,Y]]$ and $\mathfrak n=(Y,t)\subseteq K[[Y,t]]$. Since
  $p\in\mathfrak m$, the substitution defining $\psi$ converges $\mathfrak m$-adically and $\psi$ is
  a local continuous homomorphism, i.e.\ $\psi(\mathfrak n^k)\subseteq\mathfrak m^k$ for all $k$. We
  produce an inverse.

  The ring $A:=K[[Y,t]]$ is complete (for $\mathfrak n$) and hence henselian. The polynomial $F(Z):=p(Z,Y)-t\in A[Z]$ satisfies
  $F(0)=p(0,Y)-t\in\mathfrak n$ (the constant term of $p(0,Y)$ is $p(0,0)=0$), and
  $F'(0)=\partial_Xp(0,Y)$ is a unit of $A$ (its constant term is $\partial_Xp(0,0)\neq0$).
  Hensel's lemma yields $\eta\in\mathfrak n$ with $p(\eta,Y)=t$. Let
  $\varphi\colon K[[X,Y]]\to K[[Y,t]]$ be the map given by $X\mapsto\eta$, $Y\mapsto Y$; again a well-defined local continuous homomorphism since $\eta\in\mathfrak n$.

  Both $\varphi\circ\psi$ and $\psi\circ\varphi$ are continuous endomorphisms, therefore it suffices to show they fix $Y,t$ and $X,Y$ respectively. For $\varphi\circ\psi$ this follows directly since $\varphi(\psi(Y))=Y$ and $\varphi(\psi(t))=\varphi(p)=p(\eta,Y)=t$, whence $\varphi\circ\psi=\mathrm{id}$. For $\psi\circ\varphi$, we have $\psi(\varphi(Y))=Y$ so it remains to show that $\psi(\varphi(X))=X$. Note that $\psi(\varphi(X))=\psi(\eta)$. Applying $\psi$ to $p(\eta,Y)=t$ gives $p(\psi(\eta),Y)=\psi(t)=p(X,Y)$. Thus $X$ and $\psi(\eta)$ are two roots, both in $\mathfrak m$, of
  $f(Z):=p(Z,Y)-p(X,Y)\in K[[X,Y]][Z]$. Note that $f'(Z)=\partial_Zp(Z,Y)$ is a unit (constant term
  $\partial_Xp(0,0)\neq0$). By the uniqueness part of Hensel's lemma we get that $\psi(\eta)=X$, hence $\psi\circ\varphi=\mathrm{id}$.

  The last assertions follow immediately from the construction of $\varphi$.
\end{proof}

\begin{theorem}\label{thm:Q-fails-domination}
  For $K=\QQ$ and $n=2$, put $p=Y^2-X^3+X$ and $g=Y$. Then $Z_\QQ(p)\subseteq Z_\QQ(g)$,
  yet $g^m/p\notin A_2(\QQ)$ for every $m\geq 1$. Hence $A_2(\QQ)$ fails the domination property
  $\mathrm{(D)}_2$, so $A_2(\QQ)\notin\W_{\QQ,2}$.
\end{theorem}

\begin{proof}
Consider the elliptic curve $E:Y^2=X^3-X$. It is well known that $E$ is a counterexample to largeness: it has rank $0$,
  with $E(\QQ)=\{(0,0),(1,0),(-1,0)\}$, and $(0,0)$ is a smooth point. Since all $\QQ$-points of
  $E$ lie on $\{Y=0\}$, taking $p=Y^2-X^3+X$ and $g=Y$ gives
  $Z_\QQ(p)=\{(0,0),(1,0),(-1,0)\}\subseteq\{Y=0\}=Z_\QQ(g)$.
  
Since $p(0,0)=0$ and $\partial_X p(0,0)=1$, \Cref{lem:technical} provides an isomorphism
  $\varphi\colon\QQ[[X,Y]]\to\QQ[[Y,t]]$ with $\varphi(Y)=Y$ and $\varphi(p)=t$; in particular
  $\varphi$ is injective on $\QQ[X,Y]$. Now, let $v$ be the valuation of $\QQ[[Y,t]]$ into $\ZZ^2_{\mathrm{lex}}$ with $v(t)=(1,0)$
  and $v(Y)=(0,1)$ and set $w:=v\circ\varphi$ on $\QQ(X,Y)$. Being a composite of two
  $\QQ$-rational places, $w$ has residue field exactly $\QQ$; moreover $w(Y)=(0,1)\geq0$ and $w(X)=v(\eta)\geq0$, so
  $\QQ[X,Y]\subseteq\mathcal O_w$. Thus
  $w\in\mathcal S^{\mathrm{rat}}_\QQ$ and $A_2(\QQ)\subseteq\mathcal O_w$.

  Finally, $w(p)=v(t)=(1,0)$ while $w(g)=v(Y)=(0,1)$, whence, for every $m\geq1$,
  \[
    w(g^m/p)\;=\;m\,(0,1)-(1,0)\;=\;(-1,\,m)\;<\;0 .
  \]
  Therefore $g^m/p\notin\mathcal O_w\supseteq A_2(\QQ)$ for every $m\geq1$ at once: a single
  valuation witnesses the failure of $\mathrm{(D)}_2$.
\end{proof}

The proof used only two features of the pair $(\QQ,E)$: the point $(0,0)$ is a smooth point of
$E$, which allowed us to solve $p=t$ formally near it; and $Z_\QQ(p)$ is finite, which is what
lies behind the choice of $g=Y$. The next criterion, valid over any infinite
field, generalizes the previous proof. 

\begin{theorem}\label{thm:nonlarge-fails-domination}
  Let $K$ be an infinite field and $p\in K[X,Y]$ irreducible such that the curve $\{p=0\}$ has a
  smooth $K$-point and $Z_K(p)$ is finite. Then $A_2(K)$ fails the domination property
  $\mathrm{(D)}_2$; hence $A_2(K)\notin\W_K$.
\end{theorem}

\begin{proof}
  Let $a$ be a smooth $K$-point of $\{p=0\}$. As failure of $\mathrm{(D)}_2$ only asserts the
  existence of a pair, we may apply a $K$-affine automorphism (which permutes
  $\mathcal S^{\mathrm{rat}}_K$, fixes $A_2(K)$, and preserves the inclusion of zero loci) and
  assume $a=(0,0)$ with $\partial_X p(0,0)\neq0$ (swapping $X$ and $Y$ if necessary).

  We first construct $g$. Let $B:=\{b\in K:(a',b)\in Z_K(p)\text{ for some }a'\in K\}$ be the set
  of second coordinates of points of $Z_K(p)$; it is finite since $Z_K(p)$ is, and $0\in B$. Put
  \[
    g\;:=\;\prod_{b\in B}(Y-b)\;\in\;K[Y],\qquad\text{so that}\qquad
    Z_K(p)\;\subseteq\;K\times B\;=\;Z_K(g).
  \]

  Since $p(0,0)=0$ and $\partial_Xp(0,0)\neq0$, \Cref{lem:technical} provides an isomorphism
  $\varphi\colon K[[X,Y]]\to K[[Y,t]]$ with $\varphi(Y)=Y$ and $\varphi(p)=t$. Exactly as in the
  proof of \Cref{thm:Q-fails-domination}, the valuation $w:=v\circ\varphi$, with $v$ the
  valuation of $K[[Y,t]]$ given by $v(t)=(1,0)$, $v(Y)=(0,1)$ in $\ZZ^2_{\mathrm{lex}}$, lies in
  $\mathcal S^{\mathrm{rat}}_K$, so $A_2(K)\subseteq\mathcal O_w$.

  As $g$ depends only on $Y$, $\varphi(g)=g(Y)$; the factors $Y-b$ with $b\neq0$ are $v$-units,
  while $v(Y)=(0,1)$, so $w(g)=(0,1)$; and $w(p)=v(t)=(1,0)$. Hence, for every $m\geq1$,
  \[
    w(g^m/p)\;=\;m\,(0,1)-(1,0)\;=\;(-1,\,m)\;<\;0,
  \]
  so $g^m/p\notin\mathcal O_w\supseteq A_2(K)$ and $\mathrm{(D)}_2$ fails.
\end{proof}

\begin{corollary}\label{cor:nonlarge}
  If $K$ is an infinite field that is not large, then $A_2(K)\notin\W_K$.
\end{corollary}

\begin{proof}
  By Pop's plane-curve criterion for largeness \cite[Prop.\ 2.1]{pop-survey}, if $K$ is not
  large there exist an irreducible polynomial $p=Y+\sum_{i+j>1}a_{ij}X^iY^j\in K[X,Y]$ for which $Z_K(p)$ is finite and $(0,0)$ is a smooth point of the curve. We conclude by \Cref{thm:nonlarge-fails-domination}.
\end{proof}

\subsection{Largeness implies domination}\label{subsec:large-implies}

We show the converse in the following theorem for perfect fields. 

\begin{theorem}\label{thm:large-implies-Dn}
  If $K$ is a perfect large field, then $A_n(K)$ satisfies $\mathrm{(D)}_n$ for every $n\geq 1$; hence
  $A_n(K)\in\W_K$.
\end{theorem}

\begin{proof}
  We use two facts. \emph{(a)} Largeness is an elementary property: Pop's plane-curve criterion
  \cite[Prop.\ 2.1]{pop-survey} expresses it by an explicit countable scheme of first-order
  axioms. Perfectness is elementary within each fixed characteristic ($\forall x\,\exists
  y\,(y^p=x)$ in characteristic $p$; automatic in characteristic $0$), and elementary extensions
  preserve the characteristic; so every elementary extension of a perfect large field is perfect
  and large. \emph{(b)} (Kuhlmann \cite[Thm.\ 17]{kuhlmann2004}; cf.\ also
  \cite{pop-ample,jarden-patching}.) If $k$ is a \emph{perfect} large field and $F/k$ is
  \emph{any} field extension admitting a \emph{$k$-rational place}---a place $F\to
  k\cup\{\infty\}$ trivial on $k$ with residue field $k$---then $k$ is existentially closed in
  $F$. (For imperfect $k$ of positive characteristic an additional input is needed.)

  By the Rabinowitsch reduction (proof of \Cref{thm:abstract-membership}) it suffices to fix
  $p,g\in K[X]$ with $Z_K(p)\subseteq Z_K(g)$ and produce $m$ with $g^m/p\in A_n(K)$;
  since $A_n(K)=\bigcap_v\mathcal O_v$ over the $K$-rational places $v$, this means one $m$ with
  $m\,v(g)\geq v(p)$ for \emph{all} $v$ at once. The hypothesis is the first-order sentence
  $\sigma\colon\ \neg\exists x\,(p(x)=0\wedge g(x)\neq 0)$, with parameters in $K$ (the
  coefficients of $p$ and $g$), and $K\models\sigma$.

  Suppose no such $m$ exists. Then for each $r\geq 1$ there is a $K$-rational place $v_r$ of $K(X)$
  with $r\,v_r(g)<v_r(p)$. Fix a non-principal ultrafilter $\mathcal U$ on $\NN$ and form the
  ultraproduct valued field $(F^\ast,v)=\prod_r(K(X),v_r)/\mathcal U$. Each $v_r$ has residue field
  exactly $K$ and is trivial on $K$, so the residue field of $v$ is the ultrapower $k:=K^{\mathcal U}$,
  into which $K$ embeds elementarily by \L{}o\'s' theorem ($K\preceq k$)---so $k$ is perfect and
  large by \emph{(a)}---and the diagonal copy of $k$ sits inside $\mathcal O_v$ with $v|_k$ trivial. Let $x_i$ be the class of $X_i$ and set
  $\alpha:=v(g(x))$, $\beta:=v(p(x))$ in the value group $\Gamma$; as $p,g\in K[X]\subseteq\mathcal
  O_{v_r}$, $\alpha,\beta\geq 0$. For each standard $N$, $\{r:N\,v_r(g)<v_r(p)\}\supseteq\{r\geq N\}$
  is cofinite, hence lies in $\mathcal U$; so by \L{}o\'s' theorem $\beta>N\alpha$ for every standard
  $N\geq 1$.

  Let $\Delta\leq\Gamma$ be the convex subgroup generated by $\alpha$; then $\beta\notin\Delta$.
  Coarsening $v$ by $\Delta$ gives a valuation $\bar v$ with residue field $L$ and
  $\bar v(g(x))=0$, $\bar v(p(x))>0$. Each $x_i$ has $\bar v(x_i)\geq 0$, so has a residue $b_i\in L$,
  and as the residue map is a homomorphism trivial on $K$, $p(b)=0$ and $g(b)\neq 0$ in $L$.
  The trivial copy of $k$ in $\mathcal O_v$ embeds into $L$, and $v$ refines $\bar v$, so $\bar v$
  induces a place $L\to k\cup\{\infty\}$ trivial on $k$ with residue field $k$: thus $L/k$ admits a
  $k$-rational place. The system $\{p=0,\ g\neq 0\}$ has coefficients in $k$ and the solution $b$ in
  $L$; by \emph{(b)} and largeness of $k$ it has a solution $c\in k^n$, so
  $k\models\exists x\,(p(x)=0\wedge g(x)\neq 0)$. But $\sigma$ has its parameters in $K$ and
  $K\models\sigma$, so $K\preceq k$ gives $k\models\sigma$---a contradiction. Hence
  some $m$ works and $g^m/p\in A_n(K)$.
\end{proof}

\thmCharacterization*

\begin{proof}
  \textup{(i)$\Rightarrow$(ii)} is \Cref{thm:large-implies-Dn}; \textup{(ii)$\Rightarrow$(iii)}
  is trivial; \textup{(iii)$\Rightarrow$(i)} is the contrapositive of \Cref{cor:nonlarge}
  (via \Cref{thm:AK-iff-domination}). Only \textup{(i)$\Rightarrow$(ii)} uses perfectness; the
  implication \textup{(iii)$\Rightarrow$(i)}, i.e.\ $A_2(K)\notin\W_K$ for non-large $K$, holds for
  every infinite field.
\end{proof}


\begin{remark}\label{rmk:instances}
Beyond the algebraically closed cased, since all henselian fields are large \cite[\S1]{pop-survey}, 
for each of them $A_n(K)\in\W_K$ gives a Nullstellensatz: for every ideal $I\subseteq K[X]$,
\[
  J(Z_K(I))=\rt[A_n(K)]{I},
\]
a single holomorphy-ring Nullstellensatz covering all of them at once. 
In the real and $p$-adic cases, it is worth being precise about what this specialization does and does not provide. It is \emph{not always} a reproof of 
the classical theorems; it is rather an \emph{alternative} description of the vanishing ideal $J(Z_K(I))$, cast in the intrinsic language of
the rational holomorphy ring $A_n(K)$. Whether this description carries the same information as the classical
one depends on the field.

Over $\Qp$ the two descriptions coincide. Jarden and Roquette had already identified the pertinent
holomorphy ring with the localized Kochen ring $\Lambda\cdot K[X]$ \cite[Thm.\ 2.1]{jardenroquette}; their
Riemann-space viewpoint, in which each $K$-rational place has a maximal centre with residue field exactly
$K$ is the direct inspiration for our notion of $\W_{K,n}$. There the $A_n(K)$-radical description is
the classical $p$-adic Nullstellensatz, and we add nothing new.

Over $\RR$, by contrast, the two descriptions are genuinely different presentations of the same ideal. The
$A_n(\RR)$-radical $\rt[A_n(\RR)]{I}$ and the real (sum-of-squares) radical
$\rt[\mathrm{real}]{I}=\{g:\exists\,m\geq1,\ \sigma\in\sum\RR[X]^2,\ g^{2m}+\sigma\in I\}$ of Krivine and
Stengle \cite{krivine,stengle} (see \cite[\S4.1]{prestel-delzell}) both equal $J(Z_\RR(I))$, hence equal one
another; but the two identifications seem not to be at the same level. One direction is formal: given a
certificate $g^{2m}+\sigma=\sum_jc_jf_j$ with $g\neq0$ (the case $g=0$ being trivial), setting
$\omega_j:=c_j\,g^{2m}/(g^{2m}+\sigma)$ gives $g^{2m}=\sum_j\omega_jf_j$, and each $\omega_j$ lies in
$A_n(\RR)$: every $\varphi\in\mathcal S^{\mathrm{rat}}_\RR$ has formally real residue field, 
so its valuation ring is convex with respect to some ordering of $\RR(X)$ (Baer--Krull), and
$0\leq g^{2m}\leq g^{2m}+\sigma$ in that ordering forces $v(g^{2m})\geq v(g^{2m}+\sigma)$---sums of
squares suffer no cancellation at real places. Hence $\rt[\mathrm{real}]{I}\subseteq\rt[A_n(\RR)]{I}$,
and the Krivine--Stengle theorem implies our theorem $A_n(\RR)\in\W_{\RR,n}$. However we do not know how to derive their theorem from ours. Our contribution over $\RR$ is thus an alternative characterization of the real radical through
largeness, not a new route to the the real Nullstellensatz. This raises the natural question: 
\begin{quote}
\emph{Is there a simple proof of the Krivine--Stengle real Nullstellensatz assuming $A_n(K)\in\W$?}
\end{quote}

Over $\ct$ the analogous explicit step is available: $A_n(\ct)$ admits a description through an analogue of
the Kochen operator---the $t$-Kochen operator $\gamma(z)=z/(1+tz^2)$, whose denominators $1+tz^2\equiv1\pmod
{(t)}$ are units at every $K$-rational place---yielding an identification $A_n(\ct)=\Lambda_L\cdot K[X]$ and a
$\ct$-Nullstellensatz. This has been carried out by the authors in \cite{cubides-uribe}.
\end{remark}


\section{Order structure of the class $\W$}\label{sec:structure}

In this section, we study the class $\W_{K,n}$ as a poset under inclusion. Whenever $n=1$ or $K$ is algebraically closed, the
poset collapses to a singleton: the only ring computing the Nullstellensatz is $E_n(K)$. In the remaining
cases, that is, for $n\geq2$ and $K$ not algebraically closed, the poset is quite wild. We show it has no
minimal elements (not even after restricting to integrally closed rings) and is not upward directed, hence has no maximum. On the positive side it is closed under unions of chains, so every member lies
below a maximal element. We begin with a proposition gathering some basic properties

\begin{proposition}\label{prop:basic-properties}
  Let $K$ be a field, $n\geq 1$ and $A\in\W_{K,n}$.
  \begin{enumerate}
    \item[(1)] $\W_{K,n}$ is convex, that is, if $A\subseteq C\subseteq B$ are subrings of $K(X)$ with
      $A,B\in\W_{K,n}$, then $C\in\W_{K,n}$.
    \item[(2)] For all $p,g\in K[X]$ with $p\neq 0$ and $Z_K(p)\subseteq Z_K(g)$, there is
      $m\geq 1$ with $g^m/p\in A$.
    \item[(3)] $\mathcal R_n(K)\subseteq A$.
    \item[(4)] For all $g\in K[X]$ with $Z_K(g)\neq\emptyset$, $1/g\notin A$. In particular,
      $K(X)\notin\W_{K,n}$.
  \end{enumerate}
\end{proposition}

\begin{proof}
  (1) Monotonicity of the $A$-radical gives $\rt[A]{I}\subseteq\rt[C]{I}\subseteq\rt[B]{I}$ for
  every ideal $I\subseteq K[X]$; since $\rt[A]{I}=J(Z_K(I))=\rt[B]{I}$, equality holds throughout.

  (2) Since $Z_K((p))=Z_K(p)\subseteq Z_K(g)$, we have $g\in J(Z_K((p)))=\rt[A]{(p)}$, so
  $g^m=\omega p$ for some $m\geq 1$ and $\omega\in A$, i.e.\ $g^m/p=\omega\in A$.

  (3) This is the special case $g=1$ of (2): if $p\in K[X]$ satisfies $Z_K(p)=\emptyset$, then
  $Z_K(p)\subseteq Z_K(1)$ holds trivially, so $1/p=1^m/p\in A$. Together with
  $K[X]\subseteq A$, this gives $\mathcal R_n(K)\subseteq A$.

  (4) If $1/g\in A$, then $1=(1/g)\,g$ exhibits $1\in\rt[A]{(g)}=J(Z_K((g)))$, which is
  impossible: $1$ does not vanish on the nonempty set $Z_K(g)$. For the last assertion, if
  $K(X)\in\W_{K,n}$, then applying the above to $g=X_1$ would give $1/X_1\notin K(X)$---a
  contradiction.
\end{proof}

\subsection{The collapsing cases}\label{subsec:rigid}

We start by an easy but useful observation.

\begin{lemma}\label{lem:numerator-vanishing}
  Let $A\in\W_{K,n}$ and suppose $f/g\in A$ with $f,g\in K[X]$ and
  $\gcd(f,g)=1$. Then $f\in J(Z_K(g))$.
\end{lemma}

\begin{proof}
  From $f=(f/g) g\in (g)A$ we get $f\in\rt[A]{(g)}$, and $\rt[A]{(g)}=J(Z_K((g)))$ as
  $A\in\W_{K,n}$.
\end{proof}

\begin{proposition}\label{prop:rigid}
  Let $K$ be an infinite field and $n\geq1$. If $n=1$ or $K$ is algebraically closed, then
  \[
    \W_{K,n}=\{E_n(K)\}.
  \]
  Explicitly, $\W_{K,1}=\{\mathcal R_1(K)\}$, and $\W_{K,n}=\{K[X]\}$ for $K$ algebraically
  closed.
\end{proposition}

\begin{proof}
  In both cases $E_n(K)\in\W_{K,n}$ by \Cref{thmA}; recall that $E_1(K)=\mathcal R_1(K)$
  by \Cref{prop:S-in-Ev} and that $E_n(K)=K[X]$ for $K$ algebraically closed
  (\Cref{thm:Ev-equals-poly}). Conversely, let $A\in\W_{K,n}$ and $f/g\in A$ with
  $\gcd(f,g)=1$.

  Suppose first $n=1$. If $a\in Z_K(g)$, then $f(a)=0$ by \Cref{lem:numerator-vanishing}, so
  $X_1-a$ divides both $f$ and $g$, contradicting $\gcd(f,g)=1$. Hence $Z_K(g)=\emptyset$ and
  $f/g\in\mathcal R_1(K)$. Thus $A\subseteq\mathcal R_1(K)$, while
  $\mathcal R_1(K)\subseteq A$ by \Cref{prop:basic-properties}(3).

  Suppose now that $K$ is algebraically closed and that $g$ is non-constant. By
  \Cref{lem:numerator-vanishing} and Hilbert's Nullstellensatz,
  $f\in J(Z_K(g))=\sqrt{(g)}$, so any irreducible factor of $g$ divides $f^m$ for some $m$,
  hence divides $f$---contradicting $\gcd(f,g)=1$. Thus $g$ is constant and $f/g\in K[X]$;
  so $A\subseteq K[X]$, and $K[X]\subseteq A$ by definition.
\end{proof}

\subsection{There are no minimal elements}\label{subsec:trichotomy}

In view of \Cref{prop:rigid}, \emph{for the remainder of this section we assume that $K$ is an
infinite field which is not algebraically closed, and that $n\geq2$.}

\begin{definition}\label{def:Wint}
  $\W^{\mathrm{int}}_{K,n}:=\{A\in\W_{K,n}:A\text{ is integrally closed in }K(X)\}$.
\end{definition}

We start by a lemma---which might be folklore---that will be used several times.

\begin{lemma}\label{lem:witness-combination}
  Let $K$ be infinite and not algebraically closed, and let $I\subseteq K[X]$ be a nonzero ideal.
  Then there exists $\sigma\in I$, $\sigma\neq 0$, with $Z_K(\sigma)=Z_K(I)$.
\end{lemma}

\begin{proof}
  Write $I=(f_1,\dots,f_r)$ with all $f_j\neq 0$ (discarding zero generators). Fix a non-constant
  $g(Y)=a_0+a_1Y+\cdots+a_mY^m\in K[Y]$ without roots in $K$ and recall the homogenization
  $h_{g_1,g_2}=a_0g_2^m+a_1g_2^{m-1}g_1+\cdots+a_mg_1^m\in(g_1,g_2)$ from the proof of the Witness
  \Cref{lem:witness}; the case analysis there shows
  $Z_K(h_{g_1,g_2})=Z_K(g_1)\cap Z_K(g_2)$ for \emph{arbitrary} $g_1,g_2\in K[X]$.
  Moreover $h_{g_1,g_2}\neq 0$ whenever $g_2\neq 0$: otherwise $g(g_1/g_2)=0$ in $K(X)$, so
  $g_1/g_2$ would be algebraic over $K$, hence in $K$ ($K$ is relatively algebraically closed in the purely
  transcendental extension $K(X)$), a root of $g$ in $K$, a contradiction. For $r=1$ take
  $\sigma=f_1$. Otherwise iterate as in \Cref{lem:witness}: $\sigma_1:=h_{f_1,f_2}$ and
  $\sigma_l:=h_{\sigma_{l-1},f_{l+1}}$; each step is nonzero (the second argument $f_{l+1}$ is
  nonzero), and the induction there gives
  $Z_K(\sigma_{r-1})=\bigcap_jZ_K(f_j)$. Take $\sigma:=\sigma_{r-1}\in I$.
\end{proof}

\begin{convention}\label{conv:5}
For the remainder of the section fix $K$ infinite and not algebraically closed, $n\geq 2$, an
element $\alpha$ algebraic over $K$ of degree $d\geq 2$ with minimal polynomial $\mu(Y)=\sum_{i=0}^dc_iY^i$ ($c_d=1$, and $c_0=\mu(0)\neq0$)
in $K[Y]$, and let
\begin{equation}\label{eq:def-q}
  q\;:=\;X_1^d\,\mu(X_2/X_1)\;\in\;K[X_1,X_2]\subseteq K[X]
\end{equation}
be its homogenization, as in \Cref{thm:S-equals-Ev-iff-acf}. Thus
$Z_K(q)=\{(0,0)\}\times K^{n-2}$ (if $q(a,b,\ast)=0$ with $a\neq 0$ then $\mu(b/a)=0$ with
$b/a\in K$, impossible as $\mu$ is irreducible of degree $\geq 2$; if $a=0$ then $b^d=0$, as $\mu$
is monic), and in particular $X_1\in J(Z_K((q)))$. Write $L:=K(\alpha)$ and let $k_0$ be the prime
field of $K$. 
\end{convention}

\begin{lemma}\label{lem:transc-series} We work over $k_0[[t]]$. 
  \begin{enumerate}
    \item[(i)] There is a family $\{e_j\}_{j\geq 2}\cup\{f_3,\dots,f_n\}\subseteq k_0[[t]]$,
      algebraically independent over $k_0(t)$, with $\operatorname{ord}_t(e_j)=j$ and
      $\operatorname{ord}_t(f_i)=2$.
    \item[(ii)] Any family in $k_0[[t]]$ that is algebraically independent over $k_0(t)$ remains
      algebraically independent over $k_1(t)$---inside $k((t))$---for every field extension $k_1\supseteq k_0$.
  \end{enumerate}
\end{lemma}

\begin{proof}
  (i) Since the algebraic closure of $k_0(t)$ is countable, but $k_0((t))$ (and any closed ball with respect to the $t$-adic valuation) has cardinality $2^{\aleph_0}$, one can inductively pick elements which remain algebraically independent. 

  (ii) The fields $k_0((t))$ and $k_1$ are linearly disjoint over $k_0$ inside $k_1((t))$. Since linearly disjoint
  extensions are free \cite[Ch.~VIII, Prop.~3.2]{lang-algebra}, a family
  $\{t,g_1,\dots,g_r\}$ remains
  algebraically independent over $k_1$ since it is algebraically independent over $k_0$, as $g_1,\dots,g_r$ are
  algebraically independent over $k_0(t)$ and $t$ is transcendental over $k_0$. In particular, $g_1,\dots,g_r$ are algebraically
  independent over $k_1(t)$.
 \end{proof}

\begin{definition}\label{def:arcs}
  For $j\geq 2$ let $\iota_j\colon K[X]\to L[[t]]$ be the $K$-algebra homomorphism
  \[
    X_1\mapsto t,\qquad X_2\mapsto \alpha t+e_j(t),\qquad X_i\mapsto f_i(t)\quad(3\leq i\leq n).
  \]
\end{definition}

\begin{lemma}\label{lem:arc-properties}
  Let $j\geq 2$.
  \begin{enumerate}
    \item[(i)] The map $\iota_j$ is injective. It extends to an embedding $K(X)\into L((t))$, and
      $w_j:=\operatorname{ord}_t\circ\,\iota_j$ is a valuation of $K(X)$, trivial on $K$, with
      value group $\ZZ$ and $K[X]\subseteq\mathcal O_{w_j}$.
    \item[(ii)] For $h\in K[X]$: $w_j(h)>0$ if and only if $h(0)=0$. In particular, $\mathfrak m_{w_j}\cap K[X]=(X_1,\dots,X_n)$.
    \item[(iii)] $w_j(q)=m_\alpha\,j+(d-m_\alpha)\geq j+1$, where $m_\alpha\geq 1$ is the
      multiplicity of $\alpha$ as a root of $\mu$.
  \end{enumerate}
\end{lemma}

\begin{proof}
  (i) For injectivity, suppose that there is $P\in K[X]\setminus\{0\}$ such that $\iota_j(P)=0$, that is, $P(t,\alpha t+e_j,f_3,\dots,f_n)=0$. Since $t$ is transcendental over $K$, $P(t,X_2,\dots,X_n)$ is nonzero over $K(t)$, and the substitution
  $X_2=\alpha t+Y_2$ (an automorphism of $L(t)[X_2,\dots,X_n]$) keeps
  $Q(Y_2,X_3,\dots,X_n):=P(t,\alpha t+Y_2,X_3,\dots,X_n)$ nonzero over $L(t)$. The vanishing
  $Q(e_j,f_3,\dots,f_n)=0$ shows that $\{e_j,f_3,\dots,f_n\}$ are algebraically dependent over
  $L(t)$, which contradicts they are algebraically independent over $k_0(t)$ by \Cref{lem:transc-series}. The last assertion is clear.

  (ii) Every generator $\iota_j(X_i)$ has $\operatorname{ord}_t\geq 1$, so the constant term of
  $\iota_j(h)$ is $h(0)$.

  (iii) Over the splitting field $F$ of $\mu$ one has $q=\prod_{i=1}^d(X_2-\alpha_iX_1)$ with
  $\alpha_1,\dots,\alpha_d$ the roots of $\mu$ listed with multiplicity. Extend $\iota_j$ to the
  $F$-algebra homomorphism $\iota_j\colon F[X]\to F[[t]]$ given by the same assignment of the
  variables; it trivially restricts to $\iota_j$ on $K[X]$ and $\operatorname{ord}_t$ on $F[[t]]$ restricts to $\operatorname{ord}_t$ on
  $L[[t]]$. Then
  $\iota_j(q)=\prod_{i}\bigl((\alpha-\alpha_i)t+e_j\bigr)$. Computing the valuation we have 
\begin{align*}
w_j(q)  & = w_j\big(\prod_{i}\bigl((\alpha-\alpha_i)t+e_j\bigr)\big) \\
        & = w_j\big(\prod_{\alpha=\alpha_i}\bigl((\alpha-\alpha_i)t+e_j\bigr)\big)+w_j\big(\prod_{\alpha\neq\alpha_i}\bigl((\alpha-\alpha_i)t+e_j\bigr)\big) \\
        & = m_\alpha w_j(e_j)+ (d-m_\alpha)  = m_\alpha j+ (d-m_\alpha)\geq j+1, 
\end{align*}  
since for $\alpha\neq \alpha_i$, we have $w_j\bigl((\alpha-\alpha_i)t+e_j\bigr)=1$ as $w_j(e_j)=j\geq 2$. For the last inequality: if $m_\alpha=1$, then $d-m_\alpha>0$ since $d\geq 2$ and $m_\alpha j+ (d-m_\alpha)\geq j+1$; if $m_\alpha\geq 2$, then $m_\alpha j+ (d-m_\alpha)\geq 2j>j+1$. 
\end{proof}

\begin{proposition}\label{prop:cutting}
  Let $K$ be infinite and not algebraically closed, $n\geq 1$, and $A\in\W_{K,n}$. Let $w$ be a
  rank-one valuation of $K(X)$ with $K[X]\subseteq\mathcal O_w$ such that
  $\mathfrak m_w\cap K[X]$ is the maximal ideal of a point $a\in K^n$, that is, if
  \[
    w(h)>0\iff h(a)=0\qquad(h\in K[X]).
  \]
  Then $A\cap\mathcal O_w\in\W_{K,n}$; and if $A$ is integrally closed in $K(X)$, so is
  $A\cap\mathcal O_w$.
\end{proposition}

\begin{proof}
  Write $A'=A\cap\mathcal O_w$; clearly $K[X]\subseteq A'$. If $A$ is integrally closed then so
  is $A'$, an intersection of two integrally closed subrings of $K(X)$. It remains to prove
  $\rt[A']{I}=J(Z_K(I))$ for every ideal $I\subseteq K[X]$. The inclusion $\rt[A']{I}\subseteq\rt[A]{I}=J(Z_K(I))$ follows directly since $A'
  \subseteq A$ and the operator is monotone. For the converse,  we distinguish two cases. If $I=0$, then $J(Z_K(0))=J(K^n)=0=\rt[A']{0}$, as $K$ is infinite. Otherwise, suppose that $I\neq 0$ and let $g\in J(Z_K(I))$. If $g=0$ the result is trivial, so assume $g\neq 0$. By
  \Cref{lem:witness-combination}, there is $\sigma\in I$ non-zero with $Z_K(\sigma)=Z_K(I)$. Since $g$ vanishes on
  $Z_K((\sigma))$ and $A\in\W$, we have that $J(Z_K((\sigma)))=\rt[A]{(\sigma)}$, therefore
  \[
    g^{m_0}=\alpha_0\,\sigma\qquad\text{for some $m_0\geq 1$ and some $\alpha_0\in A$}.
  \]
    We split in cases depending on the value of $g(a)$. If $g(a)\neq 0$, then $w(g)=0$. Since $g\in J(Z_K(I))$), $a\notin Z_K(I)=Z_K(\sigma)$, and thus $w(\sigma)=0$. In particular, $w(\alpha_0)=m_0\,w(g)-w(\sigma)=0$ and
  $\alpha_0\in A'$. Then $g^{m_0}=\alpha_0\sigma\in IA$, which shows $g\in\rt[A']{I}$. If
  $g(a)=0$, then $w(g)>0$; a rank-one value group is archimedean and $w(\sigma)<\infty$
  ($\sigma\neq 0$), so $k\,w(g)\geq w(\sigma)$ for some integer $k\geq 0$. Then $g^k\alpha_0\in A$ and 
   \[
   w(g^{k}\alpha_0)= w(g^k(g^{m_0}/\sigma) =  (k+m_0)w(g) - w(\sigma) \geq 0,
   \]
which shows that $g^k\alpha_0\in A'$. This shows that $g^{m_0+k}=g^k\alpha_0\sigma\in IA'$, which shows $g\in\rt[A']{I}$.
\end{proof}

\begin{theorem}\label{thm:no-minimal}
  Let $K$ be infinite and not algebraically closed, and let $n\geq 2$. For every $A\in\W_{K,n}$
  there exists $A'\in\W_{K,n}$ with $A'\subsetneq A$; if $A$ is integrally closed, so is $A'$.
  Consequently neither $\W_{K,n}$ nor $\W^{\mathrm{int}}_{K,n}$ has a minimal element---in
  particular, neither has a least element. 
\end{theorem}

\begin{proof}
  By the choice of $q$, $X_1\in J(Z_K((q)))$. Therefore, \Cref{prop:basic-properties}(2) (with
  $p=q$, $g=X_1$) yields $m\geq 1$ with $\xi:=X_1^m/q\in A$. Set $j:=m+1\geq 2$. By
  \Cref{lem:arc-properties}(iii),
  \[
    w_j(\xi)\;=\;m-w_j(q)\;\leq\;m-(j+1)\;<\;0,
  \]
  so $\xi\notin\mathcal O_{w_j}$. The valuation $w_j$ is of rank one and $\mathfrak{m}_{w_j}\cap K[X]=(X_1,\ldots, X_n)$
 by \Cref{lem:arc-properties}(i),(ii), so
  \Cref{prop:cutting} implies $A':=A\cap\mathcal O_{w_j}\in\W_{K,n}$, integrally closed
  whenever $A$ is; and $\xi\in A\setminus A'$ gives $A'\subsetneq A$.
\end{proof}

\begin{remark}\label{rmk:AK-not-minimal}
  Note that \Cref{thm:no-minimal} applies in particular to the distinguished members of
  $\W_{K,n}$: neither $E_n(K)$ nor, when $K$ is large and perfect, the rational holomorphy ring
  $A_n(K)\in\W^{\mathrm{int}}_{K,n}$ (\Cref{thmB}; $A_n(K)$ is integrally closed, being an
  intersection of valuation rings) is minimal. Applying \Cref{thm:no-minimal} to the ring $A_n(K)$
  produces integrally closed members of $\W_{K,n}$ strictly contained in the rational holomorphy
  ring. There is thus no
  smallest ring of denominators computing the closure operator $J\circ Z_K$, integrally closed or
  not.
\end{remark}

\subsection{There is no maximum and the existence of maximal elements}\label{subsec:maximal}

In this subsection, we describe the top of the poset $\W_{K,n}$, showing it is not upward directed and hence, having no maximum. We start on the positive side, showing that every element is contained in a maximal one. 



\begin{proposition}\label{prop:chains-up}
  If $(A_i)_{i\in\Lambda}$ is a chain in $\W_{K,n}$ (totally ordered by inclusion), then
  $A:=\bigcup_iA_i\in\W_{K,n}$.
\end{proposition}

\begin{proof}
  A union of a chain of subrings is a subring, and $K[X]\subseteq A\subseteq K(X)$. Fix
  $I=(f_1,\dots,f_r)$. Every relation $g^m=\sum_j\omega_jf_j$ with $\omega_j\in A$ uses finitely many
  $\omega_j$, which lie in a common $A_i$ (the family is a chain); hence
  $\rt[A]{I}=\bigcup_i\rt[A_i]{I}$. As $\rt[A_i]{I}=J(Z_K(I))$ for every $i$, also
  $\rt[A]{I}=J(Z_K(I))$.
\end{proof}

A direct application of Zorn's lemma yields the following corollary. 

\begin{corollary}\label{cor:maximals-exist}
  Every $A\in\W_{K,n}$ is contained in a maximal element of $\W_{K,n}$. \qedhere
\end{corollary}

We now exhibit two members of $\W_{K,n}$ admitting no common upper bound. The construction exploits again the homogeneous polynomial $q$ defined in \eqref{eq:def-q}.

\begin{theorem}\label{thm:no-maximum}
  Let $K$ be infinite, not algebraically closed, and $n\geq2$. Then $\W_{K,n}$ is not upward directed:
  there are $C_1,C_2\in\W_{K,n}$ with no common upper bound in $\W_{K,n}$. Consequently $\W_{K,n}$ has
  \emph{no maximum}, and by \Cref{cor:maximals-exist} at least two distinct maximal elements.
\end{theorem}

\begin{proof} We follow the notation of \eqref{conv:5} with $\mu$ irreducible
  of degree $d\geq2$ and $q$ given by \eqref{eq:def-q}. Recall that
  $Z_K(q)=\{X_1=X_2=0\}$, and note $q(0,X_2)=X_2^d$ and $q(X_1,0)=c_0X_1^d$. For
  $0\leq i\leq d-1$ set
  \[
    \omega_i:=\frac{X_1^{d-1-i}X_2^i}{q}\in K(X),\qquad
    C_1:=E_n(K)[\omega_0,\dots,\omega_{d-2}],\quad
    C_2:=E_n(K)[\omega_1,\dots,\omega_{d-1}].
  \]
  Each fraction $\omega_i$ is reduced since a common factor of $X_1^{d-1-i}X_2^i$ and $q$ would be a
  monomial, but $q$ does not vanishes identically on $X_1$ nor $X_2$. We now look at $A:=K[\omega_0,\ldots, \omega_{d-1}]$. Every element $h\in A$ is a $K$-linear combination of monomials of the form $\prod_{j=1}^m \omega_{i_j}$. For any such monomial we have that 
  \[
  \prod_{j=1}^m \omega_{i_j}= \frac{X_1^aX_2^b}{q^m}, \text{ with $b=i_1+\cdots+i_m$ and $a=m(d-1)-b$.}
  \]
    In fact, any element of form $\frac{X_1^aX_2^b}{q^d}$ with $a+b=d(d-1)$ is such a product (this is possible since
  $0\leq b\leq d(d-1)$). This shows that for every 
  $0\leq b\leq d(d-1)$). This shows that for every homogeneous polynomial $h$ in $X_1, X_2$ of degree $d(d-1)$, we have that $h/q^{d}\in A$, since $h$ is a $K$-linear combination of products of the form $X_1^aX_2^b$ with $a-b=d(d-1)$. Applying this to $h=q^{d-1}$ we have that 
  \[
    \frac1q\;=\;\frac{q^{d-1}}{q^d}\;\in\;A= K[\omega_0,\dots,\omega_{d-1}].
  \]
  Since $Z_K(q)\neq\emptyset$, \Cref{prop:basic-properties}(4) implies that no member of
  $\W_{K,n}$ contains all the $\omega_i$. This yields in turn that, if $C_1,C_2\in \W_{K,n}$, they cannot have a common refinement as it would contain $A$. So it remains to show that $C_1,C_2\in \W_{K,n}$.  

  Not that both $C_1$ and $C_2$ satisfy $\mathrm{(D)}_n$, since $E_n(K)$
  does (\Cref{cor:Ev-has-Dn}) and $\mathrm{(D)}_n$ passes to any overring. Therefore, by \Cref{thm:abstract-membership}, we only need to verify that for every $a\in K^n$, there is a permutation $\sigma_a$ such that $C_1\subseteq \mathcal{O}_{P_{a,\sigma_a}}$ (resp. $C_2\subseteq \mathcal{O}_{P_{a,\sigma_a}}$). Let $a=(a_1,\ldots, a_n)\in K^n$. As each $\omega_i$ depends only on $X_1,X_2$,
  only the evaluations of these two variables matter. We split in cases depending on the values of $a_1,a_2$. If $(a_1,a_2)\neq(0,0)$ then
  $q(a_1,a_2)\neq 0$ by the choice of $q$. In this case we take $\sigma_a=\mathrm{id}$ since the denominator is never identically zero, so $P_{a,\sigma_a}$ is finite on every $\omega_i$.
  Now suppose that $(a_1,a_2)=(0,0)$. For $C_1$ let $\sigma_a$ be again the identify so that we evaluate $X_1$ first: the numerators of
  $\omega_0,\dots,\omega_{d-2}$ all contain a factor of $X_1$ and the denominamator becomes $X_2^d$, so
  $\omega_i\mapsto 0/X_2^d=0$. For $C_2$ let $\sigma_a$ swap $a_1, a_2$ so that we evaluate $X_2$ first: the numerators of
  $\omega_1,\dots,\omega_{d-1}$ all contain a factor od $X_2$ and the denominator becomes $c_0X_1^d$, so $\omega_i\mapsto 0/(c_0X_1^d)=0$. In
  each case we get a finite element. This shows that $C_1,C_2\in \W_{K,n}$. 

The last assertion follows by (\Cref{cor:maximals-exist}). 
\end{proof}

We have now all the tools to prove the dichotomy.

\thmTrichotomy*

\begin{proof}
 Point (1) is \Cref{prop:rigid}. For (2), the absence of minimal elements---even among
  integrally closed subrings---is \Cref{thm:no-minimal}, and the absence of a maximum is
  \Cref{thm:no-maximum}.
\end{proof}

\bibliographystyle{amsplain}
\bibliography{biblio}

\end{document}